\documentclass[12pt]{amsart}
\usepackage{amssymb, amsthm, enumerate, graphicx, amscd, amsmath, amssymb, amsfonts, xcolor}

\numberwithin{equation}{section}
\newtheorem{theorem}[equation]{Theorem}

\newtheorem{corollary}[equation]{Corollary}

\title{Nonlinear expansions in reproducing kernel Hilbert spaces}

\author{Javad Mashreghi}%
\address{
D\'{e}partement de Math\'{e}matiques et de Statistique\\ 
Universit\'{e} Laval\\ 
Qu\'ebec\\
QC\\
G1V 0A6 \\
Canada} 
\email{Javad.Mashreghi@mat.ulaval.ca}

\author{William Verreault*}
\address{
D\'{e}partement de Math\'{e}matiques et de Statistique\\ 
Universit\'{e} Laval\\ 
Qu\'ebec\\
QC\\
G1V 0A6 \\
Canada} 
\email{william.verreault.2@ulaval.ca}



\begin{document}

\begin{abstract}
We introduce an expansion scheme in reproducing kernel Hilbert spaces, which as a special case covers the celebrated Blaschke unwinding series expansion for analytic functions. The expansion scheme is further generalized to cover Hardy spaces $H^p$, $1<p<\infty$, viewed as Banach spaces of analytic functions with bounded evaluation functionals. In this setting a dichotomy is more transparent: depending on the multipliers used, the expansion of $f \in H^p$ converges either to $f$ in $H^p$-norm or to its projection onto a model space generated by the corresponding multipliers. Some explicit instances of the general expansion scheme, which are not covered by the previously known methods, are also discussed.
\end{abstract}

\maketitle

\section{Introduction}
An entire function $f$ has the Taylor series expansion
\begin{equation}\label{E:taylor-expansion}
f(z) = \sum_{n=0}^{\infty} a_n z^n,
\end{equation}
which converges for all values of $z \in \mathbb{C}$. We usually express $a_n$ either by the Brook Taylor formula (1715)
\[
a_n = \frac{f^{(n)}(0)}{n!},
\]
or by the Cauchy integral formula (1875)
\[
a_n = \frac{1}{2\pi i} \int_{\Gamma} \frac{f(z)}{z^{n+1}} \, dz.
\]
In the mid-1990s, R. Coifman had a revolutionary interpretation of the expansion \eqref{E:taylor-expansion}. The idea was developed in further depth in the doctoral thesis of M. Nahon \cite{MR2700759}.

Briefly speaking, their strategy is as follows. We may write $f(z)$ as
\begin{equation}\label{E:taylor-step1}
f(z) = f(0) + \big(f(z)-f(0)\big).
\end{equation}
The first term on the right side is the constant $f(0)=a_0$. The second term is a function that vanishes at the origin. Hence, we can write $f(z)-f(0) = z f_1(z)$, where $f_1$ is another holomorphic function. Thus, we have extracted a zero of $f(z)-f(0)$ as the multiplicative factor $z$, and then we deal with the new function $f_1$ by
iterating the above procedure. More explicitly, we write
\begin{equation}\label{E:taylor-step2}
f_1(z) = f_1(0) + \big(f_1(z)-f_1(0)\big),
\end{equation}
and $f_1(z)-f_1(0) = z f_2(z)$ for yet another holomorphic function $f_2$. Plugging back \eqref{E:taylor-step2} into \eqref{E:taylor-step1}, we see that
\[
f(z) = f(0) + f_1(0)z+z^2f_2(z),
\]
which at the same times shows $f_1(0)=a_1$. If we continue this procedure infinitely many times, we obtain
\begin{equation}\label{E:taylor-step3}
f(z) = f(0) + f_1(0)z+f_2(0)z^2+\cdots,
\end{equation}
which is a new interpretation of the Taylor series expansion \eqref{E:taylor-expansion}. The novel vision of the Coifman school was to extract {\em more} zeros in each step via finite Blaschke products. In the above procedure, instead of factoring out the zero at the origin, we may factor out all the zeros in the open unit disc $\mathbb{D}$. Since $f$ is entire, it has finitely many zeros in $\mathbb{D}$, and thus an appropriate apparatus to extract zeros in $\mathbb{D}$ are finite Blaschke products. Given $\alpha_1,\dots,\alpha_N \in \mathbb{D}$, repetition allowed, the rational function
\[
B(z) = \prod_{n=1}^{N} \frac{\alpha_n-z}{1-\bar{\alpha}_nz}
\]
is called a {\em finite Blaschke product}. For a treatment of this topic, we refer to the monographs \cite{MR3793610, MR2986324}. After the initial step \eqref{E:taylor-step1}, we factor $f(z)-f(0)$ as
\[
f(z)-f(0) = B_1(z) f_1(z),
\]
where $f_1$ is analytic on $\overline{\mathbb{D}}$ and has no roots in $\mathbb{D}$. All roots of $f(z)-f(0)$ in $\mathbb{D}$, including the root at the origin, and counting their multiplicities, are gathered in $B_1$. Hence, we can write
\[
f(z) = f(0) + B_1(z) f_1(z).
\]
Since $f_1$ is a holomorphic function on $\overline{\mathbb{D}}$, we can iterate the above procedure with $f_1(z) - f_1(0)$. The outcome, after the second step, is
\[
f(z) = f(0) + f_1(0)B_1(z) + B_1(z)B_2(z)f_2(z).
\]
If we reiterate infinitely many times, we obtain the expansion
\begin{equation}\label{coifman-expansion}
f(z) = f(0) + f_1(0)B_1(z) + f_2(0)B_1(z)B_2(z) + \cdots,
\end{equation}
which is known as the {\em Blaschke unwinding series expansion} of $f$.

The convergence of the Blaschke unwinding series was a major question for a long period. In fact, some of Nahon's numerical experiments suggested that, for functions in $H^2$, the unwinding expansion converges to the original function in mean. However, this fact was only firmly proved later by T. Qian \cite{MR2662313}. This question was again studied by Coifman and Steinerberger in \cite{MR3685990}, where they obtained more general results for convergence in weighted subspaces of $H^2$. Eventually, Coifman and Peyri\`{e}re \cite{MR3953482} proved the convergence of the inner-outer unwinding series for functions in the Hardy spaces $H^p$.

In parallel and independently, Qian and collaborators came up with a similar process to obtain a nonlinear phase unwinding of holomorphic functions, which they called \textit{adaptive Fourier decomposition}  \cite{MR2662313, MR2776445, MR2907888}. Variants of this type of unwinding have been further investigated in several papers, with a strong emphasis on the applications. Their idea is to start with the Takenaka--Malmquist--Walsh basis. Then, to decompose a function $f\in H^2$ with respect to this basis, they use an adaptive algorithm akin to a greedy algorithm which relies on the existence of a point in $\mathbb{D}$ that minimizes the distance from $f$ to the partial sums of its unwinding series.

In this work, we provide a general expansion scheme, which can be viewed as a generalization of the Blaschke unwinding series expansion \eqref{coifman-expansion}. We perform this in two contexts. First, in Section \ref{S:expansion-rkhs}, we present a scheme followed by two convergence results in Theorems \ref{T:conv-1-rkhs} and \ref{T:conv-2-rkhs}, in the general setting of reproducing kernel Hilbert spaces (RKHS). To some extent, these results can be extended to reproducing kernel Banach spaces. However, since in this note, our goal is to extend the above approach to Hardy spaces, we devoted the subsequent section to a more detailed study of $H^p$ spaces. More explicitly, the construction is as follows. We consider a sequence $(b_n)_{n \geq 1}$ of elements in the closed unit ball of $H^\infty$ (not necessarily Blaschke products, or even inner functions). Then using the co-analytic Toeplitz operator $T_{\bar{b}_k}$ and the closely related operators $Q_{b_k}$, we introduce the expansion
\[
f = Q_{b_1}f + b_1 \cdot Q_{b_2}T_{\bar{b}_1}f + b_1b_2 \cdot Q_{b_3}T_{\bar{b}_1\bar{b}_2}f + b_1b_2b_3 \cdot Q_{b_4}T_{\bar{b}_1\bar{b}_2\bar{b}_3}f +\cdots
\]
for an arbitrary element $f \in H^p$. The procedure is explained in Section \ref{S:expansion} and the convergence problem is addressed in Theorems \ref{T:conv-1} and \ref{T:conv-2}. In Section \ref{S:taylor}, we show that as a very special case, Theorem \ref{T:conv-1} leads to the Taylor series expansion. In Section \ref{Blaschke-section}, we show that when the $b_k$s are Blaschke factors, the general expansion scheme gives the previously known expansions which were described above, namely the Blaschke unwinding and the adaptive Fourier decomposition. In Section \ref{S:outer-expansion}, as a prototypical example, we study a special expansion created by outer functions, which was not possible with previously known expansions.

\section{The nonlinear expansion in RKHS} \label{S:expansion-rkhs}
Let $\mathcal{H}$ be an RKHS on $X$ with the multiplier algebra $\mathcal{M}(\mathcal{H})$. For a thorough treatment of RKHS, see  \cite{MR3526117}. Let $\phi \in \mathcal{M}(\mathcal{H})$ and define
\begin{equation}\label{E:def-PQ}
P_{\phi} := M_{\phi} M_{\phi}^*
\quad\mbox{and}\quad
Q_{\phi} :=I- M_{\phi} M_{\phi}^*,
\end{equation}
where $M_\phi(f)=\phi f$ is the multiplication operator on $\mathcal{H}$. While $P_{\phi}$ and $Q_{\phi}$ are certainly bounded operators, in general, they are not necessarily projections (idempotent). As a matter of fact, it is easy to see that they are projections if and only if
\[
M_{\phi}^{*}(\phi k_x) = k_x, \qquad x \in X,
\]
where $k_x$ is the reproducing kernel of $\mathcal{H}$.

The scheme is as follows. Trivially $Q_{\phi} + P_{\phi} = I$.  Hence, we have
\[
g = Q_{\phi}g + P_{\phi}g, \qquad g \in \mathcal{H},
\]
or equivalently
\begin{equation}\label{E:step0-rkhs}
g = Q_{\phi}g + \phi M_{\phi}^{*}g, \qquad g \in \mathcal{H}.
\end{equation}
This elementary observation is actually the building block in generalizing the Blaschke unwinding series expansion \eqref{coifman-expansion}.

To continue, let $(\phi_n)_{n \geq 1}$ be a sequence of elements in the closed unit ball of $\mathcal{M}(\mathcal{H})$. Fix $f \in \mathcal{H}$. Then, by \eqref{E:step0-rkhs} with $\phi=\phi_1$,
\begin{equation}\label{E:step1-rkhs}
f = Q_{\phi_1}f + \phi_1 M_{\phi_1}^{*}f.
\end{equation}
Then we apply \eqref{E:step0-rkhs} with $\phi=\phi_2$ to $g=M_{\phi_1}^{*}f$ to obtain
\begin{equation}\label{E:step2-rkhs}
M_{\phi_1}^{*}f = Q_{\phi_2}M_{\phi_1}^{*}f + \phi_2 M_{\phi_2}^{*}M_{\phi_1}^{*}f.
\end{equation}
If we plug \eqref{E:step2-rkhs} back into \eqref{E:step1-rkhs}, it gives
\begin{equation}\label{E:step3-rkhs}
f = Q_{\phi_1}f + \phi_1 Q_{\phi_2}M_{\phi_1}^{*}f + \phi_1\phi_2 M_{\phi_1\phi_2}^{*}f.
\end{equation}
Note that we implicitly used $M_{\phi_2}^{*}M_{\phi_1}^{*} = M_{\phi_1\phi_2}^{*}$. We continue and apply again \eqref{E:step0-rkhs} with $\phi=\phi_3$ and $g=M_{\phi_1\phi_2}^{*}f$ and plug it back into \eqref{E:step3-rkhs} to obtain
\begin{equation}\label{E:step4-rkhs}
f = Q_{\phi_1}f + \phi_1 Q_{\phi_2}M_{\phi_1}^{*}f + \phi_1\phi_2 Q_{\phi_3}M_{\phi_1\phi_2}^{*}f + \phi_1\phi_2\phi_3 M_{\phi_1\phi_2\phi_3}^{*}f.
\end{equation}
By induction, we can continue this procedure as many times as we wish. The general convergence theorem is as follows.

\begin{theorem} \label{T:conv-1-rkhs}
Let $\mathcal{H}$ be an RKHS on $X$ with the multiplier algebra $\mathcal{M}(\mathcal{H})$. Let $(\phi_n)_{n \geq 1}$ be a sequence of elements in the closed unit ball of $\mathcal{M}(\mathcal{H})$. Write $\Phi_0=1$ and
\[
\Phi_n := \phi_1\phi_2\cdots \phi_n, \qquad n \geq 1.
\]
Assume that
\[
\lim_{n \to \infty} \Phi_n(x) = 0, \qquad x \in X.
\]
Then, for each $f \in \mathcal{H}$,
\[
f = \sum_{n=1}^{\infty} \Phi_{n-1} \cdot Q_{\phi_n} M_{\Phi_{n-1}}^{*}f,
\]
where the series converges in $\mathcal{H}$.
\end{theorem}

\begin{proof}
By induction, the general formula for \eqref{E:step3-rkhs} and \eqref{E:step4-rkhs} is
\[
f = \sum_{n=1}^{N} \Phi_{n-1} Q_{\phi_n}M_{\Phi_{n-1}}^{*}f + \Phi_{N} M_{\Phi_{N}}^{*}f.
\]
Since
\[
\|\Phi_{N} M_{\Phi_{N}}^{*} f\|_{\mathcal{H}} \leq \|\Phi_{N}\|_{\mathcal{M}(\mathcal{H})} \|M_{\Phi_{N}}^{*} f\|_{\mathcal{H}} \leq \|M_{\Phi_{N}}^{*} f\|_{\mathcal{H}},
\]
it is enough to show that
\[
\|M_{\Phi_{N}}^{*} f\|_{\mathcal{H}} \to 0
\]
as $N \to \infty$.

Recall that
\[
M_{\Phi_{N}}^{*} k_{x} = \overline{\Phi_N(x)}\, k_x, \qquad x \in \mathcal{H}.
\]
Hence,
\[
\|M_{\Phi_{N}}^{*} k_{x}\|_{\mathcal{H}} = |\Phi_N(x)| \, \|k_{x}\|_{\mathcal{H}}.
\]
Therefore, according to our main assumption,
\begin{equation}\label{E:lim-kx=0}
\lim_{N \to \infty} \|M_{\Phi_{N}}^{*} k_{x}\|_{\mathcal{H}} =0, \qquad x \in X.
\end{equation}

For a general $f \in \mathcal{H}$, we use two properties of reproducing kernels and multipliers of $\mathcal{H}$. First, the linear span of the kernel functions is dense in $\mathcal{H}$. Second, the operators $M_{\Phi_{N}}^{*}$ are uniformly bounded. More explicitly, given $f \in \mathcal{H}$ and $\varepsilon>0$, there are constants $\alpha_1,\dots,\alpha_m \in \mathbb{C}$ and points $x_1,\dots,x_m \in X$ such that
\[
\|f-(\alpha_1 k_{x_1} + \cdots + \alpha_m k_{x_m})\|_{\mathcal{H}} < \varepsilon.
\]
Hence,
\begin{eqnarray*}
\|M_{\Phi_{N}}^{*}f\|_{\mathcal{H}}
&\leq& \|M_{\Phi_{N}}^{*}[f-(\alpha_1 k_{x_1} + \cdots + \alpha_m k_{x_m})]\|_{\mathcal{H}}\\
&+& \|M_{\Phi_{N}}^{*}(\alpha_1 k_{x_1} + \cdots + \alpha_m k_{x_m})\|_{\mathcal{H}}\\
&\leq& \|\Phi_{N}\|_{\mathcal{M}(\mathcal{H})} \|f-(\alpha_1 k_{x_1} + \cdots + \alpha_m k_{x_m})\|_{\mathcal{H}}\\
&+& |\alpha_1| \, \|M_{\Phi_{N}}^{*}k_{x_1}\|_{\mathcal{H}} + \cdots + |\alpha_m| \,  \|M_{\Phi_{N}}^{*}k_{x_m}\|_{\mathcal{H}}\\
&\leq&  \varepsilon+ |\alpha_1| \, \|M_{\Phi_{N}}^{*}k_{x_1}\|_{\mathcal{H}} + \cdots + |\alpha_m| \,  \|M_{\Phi_{N}}^{*}k_{x_m}\|_{\mathcal{H}}.
\end{eqnarray*}
By \eqref{E:lim-kx=0}, we see that
\[
\limsup_{N \to \infty} \|M_{\Phi_{N}}^{*}f\|_{\mathcal{H}}  \leq \varepsilon.
\]
Since $\varepsilon>0$ is arbitrary, the result follows.
\end{proof}

In Theorem \ref{T:conv-1-rkhs}, we assume that the multipliers are arranged so that
\[
\lim_{n \to \infty} \Phi_n(x) = 0, \qquad x \in X.
\]
In general, this is not necessarily the case. In fact, quite often we end up with
\[
\lim_{n \to \infty} \Phi_n(x) = \Phi(x), \qquad x \in X,
\]
where $\Phi$ is a non-zero multiplier of $\mathcal{H}$. In this case, an extra term appears which is linked to the model spaces (see next section). In the general setting, the result is as follows.

\begin{theorem} \label{T:conv-2-rkhs}
Let $\mathcal{H}$ be an RKHS on $X$ with the multiplier algebra $\mathcal{M}(\mathcal{H})$. Let $(\phi_n)_{n \geq 1}$ be a sequence of elements in the closed unit ball of $\mathcal{M}(\mathcal{H})$. Write $\Phi_0=1$ and
\[
\Phi_n := \phi_1\phi_2\cdots \phi_n, \qquad n \geq 1.
\]
Assume that there is a multiplier $\Phi \in \mathcal{M}(\mathcal{H})$ such that
\[
\lim_{n \to \infty} \Phi_n(x) = \Phi(x), \qquad x \in X.
\]
Then, for each $f \in \mathcal{H}$,
\[
f = \Phi M_{\Phi}^{*}f + \sum_{n=1}^{\infty} \Phi_{n-1}\cdot  Q_{\phi_n} M_{\Phi_{n-1}}^{*}f,
\]
where the series converges in $\mathcal{H}$.
\end{theorem}

\section{The expansion in Hardy spaces} \label{S:expansion}
The Hardy spaces $H^p$, $1<p<\infty$, $p \neq 2$, are not Hilbert spaces. However, with appropriate adjustments, the expansions discussed in Theorems \ref{T:conv-1-rkhs} and \ref{T:conv-2-rkhs} can be extended to this setting. For this purpose, some facts from the theory of Toeplitz operators are needed. For basics of Toeplitz operators, see \cite[Ch. 4]{MR3526203}, and for the theory of Hardy spaces, we refer to \cite{Duren, MR2500010}.

Let $\varphi \in L^\infty(\mathbb{T})$, and let $1<p<\infty$. Then, the Toeplitz operator $T_{\varphi}: H^p \to H^p$ is defined by
\[
T_{\varphi}f = P_+(\varphi f), \qquad f \in H^p,
\]
where $P_+$ is the M. Riesz projection of $L^p$ onto $H^p$. As a consequence of a celebrated result of M. Riesz on the boundedness of the Hilbert transform, $T_{\varphi}$ is also bounded and moreover
\begin{equation}\label{E:toeplitz-0}
\|T_{\varphi}\|_{H^p \to H^p} \leq c_p \|\varphi\|_{H^\infty}.
\end{equation}
It is well-known that if $\varphi$ and $\psi$ are in $H^\infty$, then
\begin{equation}\label{E:toeplitz-1}
T_{\varphi}f = \varphi f, \qquad f \in H^p,
\end{equation}
and
\begin{equation}\label{E:toeplitz-2}
T_{\bar{\varphi}} T_{\bar{\psi}} = T_{\bar{\psi}} T_{\bar{\varphi}} = T_{\bar{\varphi}\bar{\psi}}.
\end{equation}

Let $b$ be an element of the closed unit ball of $H^\infty$. As in \eqref{E:def-PQ}, let
\[
P_b := T_b T_{\bar{b}}
\quad\mbox{and}\quad
Q_b:=I-P_b.
\]
Here, the bounded operators $P_b$ and $Q_b$ are projections if and only if $b$ is an inner function. We may trivially write
\[
f = Q_bf + P_bf, \qquad f \in H^p,
\]
or equivalently, by \eqref{E:toeplitz-1},
\begin{equation}\label{E:step0}
f = Q_bf + b T_{\bar{b}}f, \qquad f \in H^p.
\end{equation}
As before, this identity is actually the basic step in generalizing the expansion scheme to Hardy spaces.

Let $(b_n)_{n \geq 1}$ be a sequence of elements in the closed unit ball of $H^\infty$. Fix $f \in H^p$. Then, by \eqref{E:step0} with $b=b_1$,
\begin{equation}\label{E:step1}
f = Q_{b_1}f + b_1 T_{\bar{b}_1}f.
\end{equation}
Then we apply \eqref{E:step0} with $b=b_2$ to $T_{\bar{b}_1}f$ to obtain
\begin{equation}\label{E:step2}
T_{\bar{b}_1}f = Q_{b_2}T_{\bar{b}_1}f + b_2 T_{\bar{b}_2}T_{\bar{b}_1}f.
\end{equation}
If we plug back \eqref{E:step2} into \eqref{E:step1} and also use \eqref{E:toeplitz-2}, we obtain
\begin{equation}\label{E:step3}
f = Q_{b_1}f + b_1 Q_{b_2}T_{\bar{b}_1}f + b_1b_2 T_{\bar{b}_1\bar{b}_2}f.
\end{equation}
The general convergence theorem is as follows.

\begin{theorem} \label{T:conv-1}
Let $(b_n)_{n \geq 1}$ be a sequence of elements in the closed unit ball of $H^\infty$. Write $B_0=1$ and
\[
B_n := b_1b_2\cdots b_n, \qquad n \geq 1.
\]
Assume that
\[
\lim_{n \to \infty} B_n(z) = 0, \qquad z \in \mathbb{D}.
\]
Then, for each $f \in H^p$,
\[
f = \sum_{n=1}^{\infty} B_{n-1}\cdot  Q_{b_n}T_{\bar{B}_{n-1}}f,
\]
where the series converges in $H^p$-norm.
\end{theorem}

\begin{proof}
By induction, the general formula for \eqref{E:step3} is
\[
f = \sum_{n=1}^{N} B_{n-1} Q_{b_n}T_{\bar{B}_{n-1}}f + B_{N} T_{\bar{B}_{N}}f.
\]
Since
\[
\|B_{N} T_{\bar{B}_{N}}f\|_{H^p} \leq \|T_{\bar{B}_{N}}f\|_{H^p},
\]
it is enough to show that
\[
\|T_{\bar{B}_{N}}f\|_{H^p} \to 0
\]
as $N \to \infty$.

Let $k_\lambda$ denote the Cauchy kernel, i.e., $k_{\lambda}(z) = (1-\bar{\lambda}z)^{-1}$. A celebrated property of conjugate-analytic Toeplitz operators is their unusual abundance of eigenvalues and eigenvectors:
\[
T_{\bar{\varphi}}k_\lambda = \overline{\varphi(\lambda)}\, k_\lambda, \qquad \lambda \in \mathbb{D}, \, \varphi \in H^\infty.
\]
Hence, in light of \eqref{E:toeplitz-2}, for each $\lambda \in \mathbb{D}$,
\[
T_{\bar{B}_{N}} k_{\lambda} = \overline{B_N(\lambda)}\, k_\lambda,
\]
and thus
\[
\|T_{\bar{B}_{N}} k_{\lambda}\|_{H^p} = |B_N(\lambda)| \, \|k_{\lambda}\|_{H^p}.
\]
Therefore, according to our main assumption,
\[
\lim_{N \to \infty} \|T_{\bar{B}_{N}} k_{\lambda}\|_{H^p} =0.
\]

For a general $f \in H^p$, incidentally we know that the span of the Cauchy kernels is also dense in $H^p$.  Therefore, given $f \in H^p$ and $\varepsilon>0$, there are constants $\alpha_1,\dots,\alpha_m \in \mathbb{C}$ and points $\lambda_1,\dots,\lambda_m \in \mathbb{D}$ such that
\[
\|f-(\alpha_1 k_{\lambda_1} + \cdots + \alpha_m k_{\lambda_m})\|_{H^p} < \varepsilon.
\]
Hence, by \eqref{E:toeplitz-0},
\begin{eqnarray*}
\|T_{\bar{B}_{N}}f\|_{H^p}
&\leq& \|T_{\bar{B}_{N}}[f-(\alpha_1 k_{\lambda_1} + \cdots + \alpha_m k_{\lambda_m})]\|_{H^p}\\
&+& \|T_{\bar{B}_{N}}(\alpha_1 k_{\lambda_1} + \cdots + \alpha_m k_{\lambda_m})\|_{H^p}\\
&\leq& c_p \|\bar{B}_{N}\|_{H^\infty} \|f-(\alpha_1 k_{\lambda_1} + \cdots + \alpha_m k_{\lambda_m})\|_{H^p}\\
&+& |\alpha_1| \, \|T_{\bar{B}_{N}}k_{\lambda_1}\|_{H^p} + \cdots + |\alpha_m| \,  \|T_{\bar{B}_{N}}k_{\lambda_m}\|_{H^p}\\
&\leq& c_p \varepsilon+ |\alpha_1| \, \|T_{\bar{B}_{N}}k_{\lambda_1}\|_{H^p} + \cdots + |\alpha_m| \,  \|T_{\bar{B}_{N}}k_{\lambda_m}\|_{H^p}.
\end{eqnarray*}
By the preceding paragraph,
\[
\limsup_{N \to \infty} \|T_{\bar{B}_{N}}f\|_{H^p}  \leq c_p \varepsilon.
\]
Since $\varepsilon>0$ is arbitrary, the result follows.
\end{proof}

Despite the above proof resembling that of Theorem \ref{T:conv-1-rkhs}, for just a harmless multiplicative factor $c_p$ shows up at the end of the proof, it is important to note that it heavily rests upon the M. Riesz theorem on the boundedness of the Hilbert transform.

One may naturally wonder what happens if $B_N$ does not converge to zero in the topology of $\mbox{Hol}(\mathbb{D})$. In this case,
\[
B := \prod_{n=1}^{\infty} b_n
\]
is a well-defined, not identically zero, function in the closed unit ball of $H^\infty$ and a minor modification of Theorem \ref{T:conv-1} implies the following result.

\begin{theorem} \label{T:conv-2}
Let $(b_n)_{n \geq 1}$ be a sequence of elements in the closed unit ball of $H^\infty$. Write $B_0=1$ and
\[
B_n := b_1b_2\cdots b_n, \qquad n \geq 1.
\]
Assume that
\[
B := \prod_{n=1}^{\infty} b_n,
\]
where the product converges uniformly on compact subsets of $\mathbb{D}$ to the element
$B \in H^\infty$.
Then, for each $f \in H^p$,
\[
f = BT_{\bar{B}}f + \sum_{n=1}^{\infty} B_{n-1}\cdot  Q_{b_n}T_{\bar{B}_{n-1}}f,
\]
where the series converges in $H^p$.
\end{theorem}

\section{Taylor series} \label{S:taylor}

If $b(z)=z$, then the Toeplitz operator $T_z$ is the {\em unilateral forward shift} operator. By the same token, $T_{\bar{z}}$ is the {\em unilateral backward shift} operator. We denote them respectively by $\mathbf{S}$ and $\mathbf{Z}$. In fact, a rather standard notation for the backward shift is $\mathbf{B}$. However, in this note, in order to avoid the consfusion with Blaschke products, we temporarily use $\mathbf{Z}$ for the backward shift. Note that when we consider them as bounded operators on the Hardy--Hilbert space $H^2$, it is legitimate to write $\mathbf{Z}=\mathbf{S}^*$. However, for other Hardy spaces, this identity is meaningless.

Using the notation of Theorem \ref{T:conv-1}, let
\[
b_n(z)=z, \qquad n\geq 1.
\]
Then
\[
\lim_{n \to \infty} B_n(z) = \lim_{n \to \infty} z^n = 0, \qquad z \in \mathbb{D}.
\]
Moreover,
\[
Q_{b_n} = I-T_{b_n}T_{\bar{b}_n} = I-T_{z}T_{\bar{z}} = I-\mathbf{S}\mathbf{Z},
\]
and thus, for $f \in H^p$,
\[
Q_{b_n}T_{\bar{B}_{n-1}}f = (I-\mathbf{S}\mathbf{Z})\mathbf{Z}^{n-1} f = \frac{f^{(n-1)}(0)}{(n-1)!}, \qquad n \geq 1.
\]
Since $B_{n-1}(z) = z^{n-1}$, the expansion
\[
f = \sum_{n=1}^{\infty} B_{n-1}\cdot  Q_{b_n}T_{\bar{B}_{n-1}}f
\]
reduces to
\[
f(z) = \sum_{n=1}^{\infty} \frac{f^{(n-1)}(0)}{(n-1)!} \cdot z^{n-1},
\]
which is precisely the Taylor series expansion of $f$.

\section{Blaschke unwinding series} \label{Blaschke-section}
In this section and the next, we study some special cases of the expansion described in Theorems \ref{T:conv-1} and \ref{T:conv-2}. Moreover, the decomposition idea for the case treated in this section is borrowed from \cite{FMN}, where the development is studied in the more general setting of $H^p$ spaces. Here, for simplicity, we just treat the case $p=2$ and provide a sketch of proofs to show that our abstract setting implies the classical Blaschke unwinding series as a special case.

Let $\lambda \in \mathbb{D}$ and let
\[
b_{\lambda}(z) := \frac{\lambda-z}{1-\bar{\lambda}z}, \qquad z \in \mathbb{D}.
\]
It is well known that $b_{\lambda}$ is an automorphism of the disc $\mathbb{D}$. In this case, $P_{b_\lambda}$ is the orthogonal projection of $H^2$ onto the Beurling subspace $b_{\lambda} H^2$, and $Q_{b_\lambda}$ is the orthogonal projection  of $H^2$ onto the model space
\[
K_{b_\lambda} = \mathbb{C} k_{\lambda}.
\]
In this case, we can provide an explicit formula for $Q_{b_\lambda}$. Given $f \in H^2$, put
\[
g := f - f(\lambda) \frac{k_{\lambda}}{k_{\lambda}(\lambda)}.
\]
Then clearly $g \in H^2$ and $g(\lambda)=0$. Hence, by a theorem of F. Riesz \cite[Page 167]{MR2500010}, $g= b_{\lambda}h$ for some $h \in H^2$. Therefore, we can write
\[
f = \frac{f(\lambda)}{k_{\lambda}(\lambda)} k_{\lambda} + b_{\lambda}h.
\]
This is precisely the orthogonal decomposition of $f$ with the first component coming from $K_{b_\lambda}$ and the second from $b_{\lambda} H^2$. According to this decomposition, we conclude that
\begin{equation}\label{E:proj-kb}
Q_{b_\lambda}f = \frac{f(\lambda)}{k_{\lambda}(\lambda)} k_{\lambda} = (1-|\lambda|^2)f(\lambda) k_{\lambda}.
\end{equation}

Now let $(\lambda_n)_{n \geq 1}$ be a sequence in $\mathbb{D}$. There are two possibilities, which are described in the following two corollaries. In each case, we are faced with the Takenaka--Malmquist--Walsh basis \cite[Ch. 5]{MR3526203}. See also \cite{MR2572653, MR1568210}.

\begin{corollary} \label{Cor:hp-blaschke}
Let $(\lambda_n)_{n \geq 1}$ be a non-Blaschke sequence in $\mathbb{D}$, i.e.,
\[
\sum_{n=1}^{\infty} (1-|\lambda_n|) = \infty.
\]
Let $B_0=1$ and
\[
B_n(z) = \prod_{k=1}^{n} \frac{\lambda_k-z}{1-\bar{\lambda}_kz}, \qquad n \geq 1.
\]
Let $f$ be entire. Then
\begin{equation}\label{E:our-B-rep}
f = f(\lambda_1) + \sum_{n=1}^{\infty} \big( (T_{\bar{B}_{n}}f)(\lambda_{n+1}) - \bar{\lambda}_n (T_{\bar{B}_{n-1}}f)(\lambda_n) \big) B_n,
\end{equation}
where the series converges in $H^2$-norm.
\end{corollary}

\begin{proof}
In this case,
\[
\lim_{n \to \infty} B_n(z) = \lim_{n \to \infty} \prod_{k=1}^{n} \frac{\lambda_k-z}{1-\bar{\lambda}_kz}  = 0, \qquad z \in \mathbb{D},
\]
and thus Theorem \ref{T:conv-1} applies. By \eqref{E:proj-kb},
\[
Q_{b_{\lambda_n}}T_{\bar{B}_{n-1}}f = (1-|\lambda_{n}|^2)(T_{\bar{B}_{n-1}}f)(\lambda_n) k_{\lambda_n}, \qquad n \geq 1.
\]
Therefore, each $f \in H^2$ has the  decomposition
\begin{equation}\label{E:dec-f-ortho1}
f = \sum_{n=1}^{\infty} (1-|\lambda_{n}|^2)(T_{\bar{B}_{n-1}}f)(\lambda_n)  B_{n-1}k_{\lambda_n}.
\end{equation}
Note that the orthonormal basis of $H^2$ in this decomposition is
\[
(1-|\lambda_{n}|^2)^{1/2}B_{n-1}k_{\lambda_n}, \qquad n \geq 1,
\]
and the corresponding Fourier coefficients of $f$ with respect to this basis are
\[
(1-|\lambda_{n}|^2)^{1/2}(T_{\bar{B}_{n-1}}f)(\lambda_n), \qquad n \geq 1.
\]

It is easy to see that $b_{\lambda}$ and $k_{\lambda}$ are related via the linear equation
\[
(1-|\lambda|^2)k_{\lambda} + \bar{\lambda} b_{\lambda} = 1.
\]
If we solve for $k_{\lambda}$ and plug it into \eqref{E:dec-f-ortho1} (with $\lambda=\lambda_n$) and recalling that $B_{n-1}b_{\lambda_n} = B_{n}$, we obtain
\[
f = \sum_{n=1}^{\infty} (T_{\bar{B}_{n-1}}f)(\lambda_n)\left( B_{n-1} - \bar{\lambda}_n B_n\right).
\]
The convergence is in $H^2$ and all the coefficients of the $B_k$s are scalars. After a rearrangement, the representation \eqref{E:our-B-rep} follows. 
\end{proof}

A remark is in order concerning the last part of the above proof. Even though it is not visible at first glance, a stronger assumption is needed to ensure the convergence after the rearrangement. That $f$ is assumed to be entire is enough for us and covers the classical setting. However, it can be slightly generalized to functions analytic on the closed unit disc. See \cite{FMN}.

The expansion \eqref{E:our-B-rep} is a very strong form of the expansion \eqref{coifman-expansion}. It is rather surprising that we do not impose heavy restrictions on $\lambda_n$. Let us explain what happens in the special case of \eqref{coifman-expansion}. Here, in the first step we set
\begin{equation}\label{E:tmp-f0-f0}
f(z)-f(0) = \mathbf{B}_1(z) f_1(z),
\end{equation}
where $\mathbf{B}_1$ is the finite Blaschke product formed with the zeros
\[
\lambda_1=0, \, \lambda_2, \, \dots, \, \lambda_N
\]
of $f(z)-f(0)$ on $\mathbb{D}$. Therefore, respecting the notation of the more general representation \eqref{E:our-B-rep}, we have
\begin{equation}\label{E:tmp-f0-f02}
\mathbf{B}_1 = B_N = b_{\lambda_1} b_{\lambda_2}\cdots b_{\lambda_N}.
\end{equation}
Note that $f_1$ is analytic on $\overline{\mathbb{D}}$ and has no roots in $\mathbb{D}$. However, the next set of zeros
\[
\lambda_{N+1} = 0, \, \lambda_{N+2}, \, \dots, \, \lambda_{M},
\]
that we choose are the zeros of $f_1(z)-f_1(0)$. As it is clear now, the origin repeats infinitely many times in this sequence and thus, after all, it is a non-Blaschke sequence. Moreover, by \eqref{E:tmp-f0-f0} and \eqref{E:tmp-f0-f02}, for $1 \leq n \leq N-1$,
\[
\bar{B}_{n} f = \bar{B}_{n}f(0) + \bar{B}_{n} \mathbf{B}_1 f_1 = \bar{B}_{n}f(0) + b_{\lambda_{n+1}} \cdots b_{\lambda_N} f_1
\]
and, for $n=N$,
\[
\bar{B}_{N} f = \bar{B}_{N}f(0) + \bar{B}_{N} \mathbf{B}_1 f_1 = \bar{B}_{N}f(0) +  f_1.
\]
Therefore,
\[
T_{\bar{B}_{n}} f =  b_{\lambda_{n+1}} \cdots b_{\lambda_N} f_1, \qquad 1 \leq n \leq N-1,
\]
and
\[
T_{\bar{B}_{N}} f =  f_1.
\]
Note that we implicitly used the fact that $\lambda_1=0$ and thus $P_{+}(\bar{B}_{n})=0$. Hence,
for $1 \leq n \leq N-1$,
\[
(T_{\bar{B}_{n}}f)(\lambda_{n+1}) - \bar{\lambda}_n (T_{\bar{B}_{n-1}}f)(\lambda_n)  = 0
\]
and, for $n=N$,
\[
(T_{\bar{B}_{N}}f)(\lambda_{N+1}) - \bar{\lambda}_N (T_{\bar{B}_{N-1}}f)(\lambda_N)  = f_1(\lambda_{N+1}) = f_1(0).
\]
In short, the expansion \eqref{E:our-B-rep} becomes
\[
f = f(0) + f_1(0) B_N + \sum_{n=N+1}^{\infty} \big( (T_{\bar{B}_{n}}f)(\lambda_{n+1}) - \bar{\lambda}_n (T_{\bar{B}_{n-1}}f)(\lambda_n) \big) B_n,
\]
which we rewrite as
\begin{equation}\label{E:rep-bwin-2}
f = f(0) + f_1(0) \mathbf{B}_1 + \sum_{n=N+1}^{\infty} \big( (T_{\bar{B}_{n}}f)(\lambda_{n+1}) - \bar{\lambda}_n (T_{\bar{B}_{n-1}}f)(\lambda_n) \big) B_n.
\end{equation}
The first two terms on the right side of \eqref{E:rep-bwin-2} are precisely the first two terms in the classical Blaschke unwinding series \eqref{coifman-expansion}. As a matter of fact, it is now easy to complete the above line of reasoning and see that in \eqref{E:rep-bwin-2}, many terms are zero and it eventually reduces to the classical setting \eqref{coifman-expansion}.

\begin{corollary} \label{Cor:hp-non-blaschke}
Let $(\lambda_n)_{n \geq 1}$ be a Blaschke sequence in $\mathbb{D}$, i.e.,
\[
\sum_{n=1}^{\infty} (1-|\lambda_n|) < \infty.
\]
Let $B_0=1$,
\[
B_n(z) = \prod_{k=1}^{n} \frac{\lambda_k-z}{1-\bar{\lambda}_kz}, \qquad n \geq 1,
\]
and
\[
B(z) = \prod_{n=1}^{\infty} \frac{|\lambda_n|}{\lambda_n} \, \frac{\lambda_n-z}{1-\bar{\lambda}_nz}.
\]
Let $f$ be entire. Then
\[
f = BT_{\bar{B}}f + \Big( f(\lambda_1) + \sum_{n=1}^{\infty} \big( (T_{\bar{B}_{n}}f)(\lambda_{n+1}) - \bar{\lambda}_n (T_{\bar{B}_{n-1}}f)(\lambda_n) \big) B_n \Big),
\]
where the series converges in $H^2$-norm.
\end{corollary}

\begin{proof}
In this case, $B$ is a well-defined infinite Blaschke product and thus Theorem \ref{T:conv-2} applies. More explicitly, each $f \in H^2$ has the  decomposition
\begin{equation}\label{E:dec-f-ortho2}
f = BT_{\bar{B}}f
+ \sum_{n=1}^{\infty} (1-|\lambda_n|^2) (T_{\bar{B}_{n-1}}f)(\lambda_n)  B_{n-1}k_{\lambda_n}.
\end{equation}
This is precisely the explicit description of the orthogonal decomposition $H^2 = BH^2 \oplus K_B$. The orthonormal basis of $K_B$ in this decomposition is
\[
(1-|\lambda_n|^2)^{1/2} B_{n-1}k_{\lambda_n}, \qquad n \geq 1,
\]
and the corresponding Fourier coefficients of the projection of $f$ onto $K_B$ are
\[
(1-|\lambda_n|^2)^{1/2} (T_{\bar{B}_{n-1}}f)(\lambda_n), \qquad n \geq 1.
\]
The rest of proof is the same as the proof of Corollary \ref{Cor:hp-blaschke}.
\end{proof}

When $p \ne 2$, the terms in the above expansions are not orthogonal. But, in technical language, Corollaries \ref{Cor:hp-blaschke} and \ref{Cor:hp-non-blaschke} give us Schauder bases for $H^p$ and $K_B$, respectively. Some other Schauder basis of rational functions (not finite Blaschke products) are presented in \cite{MR3195914} for $H^p$ spaces.

\section{Unwinding series with outer functions} \label{S:outer-expansion}
In this section, we explore the development created by the outer function
\[
b(z)=\frac{z-1}{2}, \qquad z \in \mathbb{D}.
\]
We take $b_1=b_2=\cdots=b$. Hence, clearly
\[
\lim_{n \to \infty} B_n(z) = \lim_{n \to \infty} b_1(z)b_2(z)\cdots b_n(z) = \lim_{n \to \infty} \left( \frac{z-1}{2} \right)^n = 0, \qquad z \in \mathbb{D}.
\]
Therefore, Theorem \ref{T:conv-1} applies and we have, for each $f \in H^p$,
\begin{eqnarray*}
f = \sum_{n=1}^{\infty} B_{n-1} Q_{b_n}T_{\bar{B}_{n-1}}f
= \sum_{n=1}^{\infty} \frac{Q_{b}T_{\bar{B}_{n-1}}f}{2^{n-1}} \, (z-1)^{n-1},
\end{eqnarray*}
where the series converges in $H^p$-norm. We can provide a more familiar formula for $Q_{b}$. Note that
\[
T_b = \frac{\mathbf{S}-I}{2}
\quad \mbox{and} \quad
T_{\bar{b}} = \frac{\mathbf{Z}-I}{2}.
\]
Recall the definition of $\mathbf{S}$ and $\mathbf{Z}$ from Section \ref{S:taylor}. Therefore,
\begin{eqnarray*}
Q_{b} &=& I - T_b T_{\bar{b}}\\
&=& I - \frac{\mathbf{S}-I}{2} \, \frac{\mathbf{Z}-I}{2}\\
&=& I - \frac{1}{4} (\mathbf{S}\mathbf{Z}-\mathbf{S}-\mathbf{Z}+I)\\
&=& \frac{1}{4} (2I+k_0 \otimes k_0 +\mathbf{S}+\mathbf{Z}).
\end{eqnarray*}



\section*{Declarations} 
\subsection*{Conflict of interest}
On behalf of all authors, the corresponding author states that there is no conflict of interest.

\subsection*{Competing interests}
On behalf of all authors, the corresponding author states that there are no competing interests.

\subsection*{Funding information}
This work was supported by the NSERC Discovery Grant (Canada), and graduate scholarships from NSERC and FRQNT (Quebec).

\subsection*{Author contribution}
All authors wrote the manuscript and reviewed the final version.

\bibliographystyle{plain}
\bibliography{OSC-references}

\end{document}